\documentclass[12pt]{amsart}
\usepackage[english]{babel}
\usepackage{latexsym}
\usepackage{amssymb}
\usepackage{amscd}
\usepackage[cp850]{inputenc}
\usepackage[mathscr]{eucal}

\newcommand\PB{{\mathrm{PB}}}

\newtheorem{teor}{Theorem}
\newtheorem{prop}{Proposition}

\newcommand{\N}{\mathbb{N}}

\newcommand{\aproof}{\begin{proof}}
\newcommand{\zproof}{\end{proof}}

\begin{document}
\title{On isomorphically polyhedral $\mathcal L_\infty$-spaces}

\author{Jes\'{u}s M. F. Castillo}

\address{Departamento de Matem\'aticas\\ Universidad de Extremadura\\
Avenida de Elvas\\ 06071-Badajoz\\ Spain} \email{castillo@unex.es}

\author{Pier Luigi Papini}
\address{via Martucci 19, 40136 Bologna, Italia}
 \email{pierluigi.papini@unibo.it}

\thanks{This research was partially supported by project MTM2013-45643-C2-1-P. The research of the second author has been partially supported by GNAMPA of the Instituto Nazionale di Alta Matematica.}

\maketitle

\begin{abstract} We show that there exist $\mathcal L_\infty$-subspaces of separable isomorphically polyhedral Lindenstrauss spaces that cannot be renormed to be a Lindenstrauss space.\end{abstract}

\section{Isomorphically polyhedral spaces}

A Banach space is said to be polyhedral if the closed unit ball of every finite
dimensional subspace is the closed convex hull of a finite
number of points. Polyhedrality  is a geometrical notion: $c_0$ is
polyhedral while $c$ is not. It is also an hereditary notion: every subspace of a polyhedral space is polyhedral.
The isomorphic notion associated with
polyhedrality is: A Banach space is said to be \emph{isomorphically polyhedral} if it admits a polyhedral renorming. The simplest examples of isomorphically polyhedral spaces are the $C(\alpha)$ spaces for $\alpha$ an ordinal, and their subspaces. In \cite{castpapext} we surveyed what is known about polyhedral $\mathcal L_\infty$-spaces, which can be summarized as follows:

\begin{enumerate}
\item There are polyhedral spaces which are not $\mathcal L_\infty$: indeed, any non $\mathcal L_\infty$ subspace of $c_0(\Gamma)$ --- recall from \cite{gkl} that subspaces of $c_0(\Gamma)$ are $\mathcal L_\infty$-spaces if and only if they are isomorphic to $c_0(I)$.

\item There are Lindenstrauss spaces not polyhedral: $C[0,1]$.

\item A result of Fonf \cite{fonfpre} asserts that preduals of $\ell_1$
are isomorphically polyhedral.

\item Fonf informed us \cite{fonfpers} that the result fails for $\ell_1(\Gamma)$: Kunen's compact
$\mathscr K$ provides, under {\sf CH}, a scattered, non metrizable, compact so that $C(\mathscr K)$ space has the rare property that every uncountable set of elements contains one that belongs to the closure of the convex hull of the others. And this property was used by Jim\'enez and Moreno \cite{jimore} to show that every equivalent renorming of  $C(\mathscr K)$ has only a countable number of weak*-strongly exposed points. Thus, no equivalent renorming can be polyhedral (see \cite{fpst}). At the same time $C(\mathscr K)^* = \ell_1(\Gamma)$ since $\mathscr K$ is scattered.
\item The trees $T$ for which $C(T)$ is
isomorphically polyhedral are characterized in \cite{fpst}. Thus, there are scattered
compact $K$ (not depending on {\sf CH} as it occurs with Kunen's compact) such that $C(K)$ is not isomorphically polyhedral.
\end{enumerate}

Fonf \cite{fonfpers} asked \cite[Section 4, problem 5]{castpapext} whether isomorphically polyhedral $\mathcal L_\infty$-spaces are isomorphically Lindenstrauss. The purpose of this note is to show that the answer is no.

\section{Preliminaries}

A Banach space $X$ is said to be an $\mathcal{L}_{\infty, \lambda}$-space if every finite dimensional
subspace $F$ of $X$ is contained in another finite dimensional subspace
of $X$ whose Banach-Mazur distance to the  corresponding space $\ell^n_\infty$
is at most $\lambda$. The space $X$ is said to be an $\mathcal L_\infty$-space if it is an $\mathcal L_{\infty, \lambda}$-space for some $\lambda$. The basic theory and examples of $\mathcal{L}_{\infty}$-spaces can be found
in \cite[Chapter 5]{lindtzaf-LN}. A Banach space $X$ is said to be a Lindenstrauss space if it is an isometric predual of some space $L_1(\mu)$. Lindenstrauss spaces correspond to $\mathcal L_{\infty,1^+}$-spaces. A Lindenstrauss space is an $\mathcal L_{\infty,1}$-space if and only if it is polyhedral (i.e., the unit ball of every finite dimensional subspace is a polytope) \cite[p.199]{lindtzaf-LN}.\\

A Banach space $X$ is said to have \emph{Pe\l czy\'nski's property $(V)$} if each
operator defined on $X$ is either weakly compact or an
isomorphism on a subspace isomorphic to $c_0$. Pe\l czy\'nski shows in \cite{pelcy} that $C(K)$-spaces enjoy property $(V)$, and Johnson and Zippin \cite{johnzippre} that Lindenstrauss spaces also have $(V)$.\\

Let $\alpha:A\to Z$ and $\beta:B\to Z$ be operators acting between Banach spaces.
 the pull-back space $\PB$ is defined as $\PB=\PB(\alpha,\beta)=\{(a,b)\in A\oplus_\infty B: \alpha(a)=\beta(b) \}$. It has the property of yielding a commutative diagram
\begin{equation}\label{pb-dia}
\begin{CD}
\PB@>{'}\!\beta>> A\\
@V {'}\!\alpha VV @VV \alpha V\\
B @> \beta >> Z
\end{CD}
\end{equation}
in which the arrows after primes are the restriction of the projections onto the corresponding factor.
Needless to say (\ref{pb-dia}) is minimally commutative in the sense that if the operators
${''}\!\beta: C\to A$ and ${''}\!\alpha: C\to B$ satisfy $\alpha\circ {''}\! \beta=\beta\circ {''}\!\alpha$,
then there is a unique operator $\gamma:C\to\PB$ such that
${''}\!\beta={'}\!\beta \gamma$ and ${''}\!\beta={'}\!\beta \gamma$.
Clearly, $\gamma(c)=({''}\!\beta(c), {''}\!\alpha(c))$ and  $\|\gamma\|\leq \max \{\|{''}\!\alpha\|, \|{''}\!\beta\|\}$. Quite clearly ${'}\!\alpha$ is onto if $\alpha$ is. As a consequence of this, if one has an exact sequence
\begin{equation}\label{SEX}\begin{CD}
0 @>>> Y @>\imath>> X @>\pi>> Z @>>> 0
\end{CD}
\end{equation}
and an operator $u:A\to Z$ then one can form the pull-back diagram of the couple $(\pi, u)$:
\begin{equation*}
\begin{CD}
0  @>>>  Y @>\imath>> X @>\pi>> Z @>>>0 \\
  &  & & &  @A {'}\!u AA  @AA u A
    \\
&& & & \PB  @> {'}\!\pi >> A
\end{CD}
\end{equation*}
Recalling that ${'}\!\pi$ is onto and taking $j(y)=(0,\imath(y))$, it is easily seen that the following
diagram is commutative:
\begin{equation}\label{pb-seq}
\begin{CD}
0  @>>>  Y @>\imath>> X @>\pi>> Z @>>>0 \\
  &  & @|  @A{'}\!u AA  @AA u A
    \\
0@>>> Y@>j >> \PB  @>{'}\!\pi >> A@>>>0\\
\end{CD}
\end{equation}
Thus, the lower sequence is  exact, and we shall refer to it as the pull-back sequence.
The well-known (see e.g., \cite{castgonz}) splitting criterion is: the pull-back sequence splits if and only if $u$ lifts to $X$; i.e., there is an operator $U:A\to X$ such that $\pi U=u$.\\

\section{An isomorphically polyhedral $\mathcal L_\infty$-space that is not Lindenstrauss}

\begin{teor} There is a separable isomorphically polyhedral $\mathcal L_\infty$ space that is not isomorphically Lindenstrauss. Moreover, it is a subspace of an isomorphically polyhedral Lindenstrauss space.
\end{teor}
\begin{proof} We need to recall from \cite{ccky} the existence of nontrivial exact sequences
$$
\begin{CD}
 0 @>>> C(\omega^\omega)  @>>> \Omega @>q>> c_0 @>>> 0
\end{CD}
$$in which the quotient map $q$ is strictly singular. This fact makes $\Omega$ fail Pe\l czy\'nski's property $(V)$. Since Lindenstrauss spaces share with $C(K)$-spaces Pe\l czy\'nski's property (V), the space $\Omega$ is not isomorphic to a Lindenstrauss space. Of course it is an $\mathcal L_\infty$-space since this is a  $3$-space property. Thus, our purpose is to show that there is an $\Omega$  as above that is isomorphically polyhedral.

We  recall from \cite[Section3]{ccky} the parameter $\rho_N(c_0)$, defined as the the least constant such that if $T: c_0\to \ell_\infty(\omega^N)$ is a bounded linear operator such that $\emph{dist}(Tx, C(\omega^N))\leq \|x\|$ for all $x\in c_0$ then  there is a linear map $L: c_0\to C(\omega^N)$ with $\|T - L\|\leq \rho_N(c_0)$. Theorems 3.1 and Lemma 3.2 in \cite{ccky} show that $\lim \rho_N(c_0)=+\infty$. Now we need a specific choice for each $N$: this is provided by \cite[Prop. 4.6]{ccky}: there is a bounded operator $T_N: c_0\to \ell_\infty(\omega^N)$ so that $\emph{dist}(T_Nx, C(\omega^N))\leq \|x\|$ for all $x\in c_0$ but such that if $E \subset c_0$ is a subspace of $c_0$ almost isometric to $c_0$ then $\rho_N(c_0)\leq 2\|T_N -L \|$ for any linear map $L: c_0\to C(\omega^N)$.\\

Let, for each $N$, a linear continuous operator $T_N: c_0\to \ell_\infty(\omega^N)$ as above. We form the twisted sum space $$C(\omega^N)\oplus_{T_N} c_0 = \left( C(\omega^N) \times c_0, \|\cdot\|_{T_N}\right)$$ endowed with the norm $\|(h,x)\|_{T_N}= \max \{\|h -T_Nx\|, \|x\|\}$. This yields an exact sequence
$$\begin{CD}\label{sequ}
0 @>>> C(\omega^N)@>i_N>> C(\omega^N)\oplus_{T_N} c_0  @>q_N>> c_0@>>>0\end{CD}$$
with embedding $i_N(f)=(f,0)$ and quotient map $q_N(f,x)=x$. The identity map
$id: C(\omega^N)\oplus_{T_N} c_0 \longrightarrow C(\omega^N)\oplus_\infty c_0$ is an isomorphism since
$$\|T_N\|^{-1}\|(f, x)\|_{T_N}\leq \|(f,x)\|_\infty \leq \|T_N\|\|(f, x)\|_{T_N}$$
and therefore the space $C(\omega^N)\oplus_{T_N} c_0$ is isomorphically polyhedral. We need now to use the main result in \cite{dfh} asserting that in a separable isomorphically polyhedral space every norm can be approximated by a polyhedral norm.
Let $\|\cdot\|_{P_N}$ be a polyhedral norm in $C(\omega^N)\oplus_{T_N} c_0$ that is $2$-equivalent to $\|\cdot\|_{T_N}$.

The sequence (\ref{sequ}) splits, but the norm of the projection goes to infinity with $N$: Indeed, if $$P: C(\omega^N)\oplus_{T_N} c_0\to C(\omega^N)$$ is a linear continuous projection then $P$ has to
have the form $P(f,x)= (f - Lx, 0)$, where $L: c_0\to C(\omega^N)$ is a certain linear map. Thus, if $x\in c_0$ is a norm one element,
one gets  $P(T_Nx, x)= (T_Nx - Lx, 0)$ and thus $\|T_Nx - Lx\|\leq \|P\|\|x\|$, hence $\|T_N - L\|\leq \|P\|$. The choice of $T_N$ forces
$\lim_{N\to \infty} \inf \|P\|=+\infty$. Therefore, the $c_0$-sum
$$
\begin{CD}
 0 @>>> c_0(C(\omega^N))  @>>> c_0(C(\omega^N)\oplus_{P_N} c_0) @>(q_N)>> c_0(c_0) @>>> 0
\end{CD}
$$cannot split. The space $ c_0(C(\omega^N)\oplus_{P_N} c_0) $ is isomorphically
polyhedral as any $c_0$-sum of polyhedral spaces \cite{hn}. We now define a suitable operator $\Delta$ so that when making the pull-back diagram
$$
\begin{CD}
 0 @>>> c_0(C(\omega^N))  @>>> c_0(C(\omega^N)\oplus_{P_N} c_0) @>(q_N)>> c_0(c_0) @>>> 0\\
 &&@| @AA{\delta}A @AA{\Delta}A\\
 0 @>>> c_0(C(\omega^N))  @>>> \Omega @>>q> c_0 @>>> 0\\
\end{CD}
$$the map $q$ is strictly singular. That prevents $\Omega$ from being Lindenstrauss under any equivalent renorming.

Pick as $\Delta$ the diagonal operator $c_0\to c_0(c_0)$ induced by the scalar sequence $(\rho_N(c_0)^{-1/2})\in c_0$; i.e.,$$\Delta(x) = (\rho_N(c_0)^{-1/2}x)_N.$$

Assume that $q$ is not strictly singular. Then, there is a subspace $E$ of $c_0$ and a linear bounded map $V:E\to \Omega$ so that $qV=\Delta_{|E}$. By the $c_0$ saturation and the distortion properties of $c_0$, there is no loss of generality assuming that $E$ is an almost isometric copy of $c_0$. By the commutativity of the diagram $(q_N)\delta V =\Delta_{|E}$, which in particular means that $q_N\delta V (e) = \rho_N(c_0)^{-1/2}e$ for all $e\in E$. This means that the map $\delta V$ has on $E$ the form $(L_Ne, \rho_N(c_0)^{-1/2}e)_N$ where $L_N: E\to C(\omega^N)$ is a linear map; by continuity, there is a constant $M$ so that $\|(L_Ne, \rho_N(c_0)^{-1/2}e)\| \leq M\|e\|$, which means
$$\|L_Ne - T_N \rho_N(c_0)^{-1/2}e\| \leq M\|e\|$$
and thus
$$\|\rho_N(c_0)^{1/2} L_N - T_N \| \leq M\rho_N(c_0)^{1/2}.$$

This contradicts the fact that $E=c_0$, the definition of $\rho_N(c_0)$ and the choice of $T_N$.\\

To conclude the proof, the definition of pull-back space implies that $\Omega$ is actually a subspace of $c_0(C(\omega^N)\oplus_{P_N} c_0) \oplus_\infty c_0$, hence isomorphically polyhedral.\end{proof}

Since $c_0(C(\omega^N)) \simeq C(\omega^\N)$, the space $\Omega$ above yields a twisted sum
$$
\begin{CD}
 0 @>>> C(\omega^\omega)  @>>> \Omega @>q>> c_0 @>>> 0
\end{CD}
$$
in which $q$ is strictly singular. The dual sequence
$$
\begin{CD}
 0 @>>> \ell_1  @>>> \Omega^* @>>> \ell_1 @>>> 0
\end{CD}
$$necessarily splits and thus $\Omega^*$ can be renormed to be $\ell_1$, although $\Omega$ cannot be endowed with an equivalent norm $|\cdot|$ so that $(\Omega, |\cdot|)^*=\ell_1$. Moreover, $\Omega$ is actually a subspace of the isomorphically polyhedral Lindenstrauss space  $c_0(C(\omega^N)\oplus_{P_N} c_0) \oplus c_0$.

\section{An isomorphically polyhedral $\mathcal L_\infty$ space that is not a Lindenstrauss-Pe\l czy\'nski space}

We show now that one can produce an $\mathcal L_\infty$-variation of $\Omega$ still farther from Lindenstrauss spaces.
Lazar \cite{laza} and Lindenstrauss \cite{lindmemo} showed that Lindenstrauss polyhedral spaces $X$ enjoy the property that
compact $X$-valued operator admit equal norm extensions. In \cite{castmorestud}, the authors introduce the Lindenstrauss-Pe\l czy\'nski spaces (in short $\mathscr {LP}$-spaces) as those Banach spaces $E$ such that all operators
from subspaces of $c_0$ into $E$ can be extended to $c_0$. The
spaces are so named because Lindenstrauss and Pe\l czy\'nski first
proved in \cite{lindpelc} that $C(K)$-spaces have this property.
Lindenstrauss spaces have also the property (see
\cite{lindpelc,castsuaisr}) as well as $\mathcal L_\infty$-spaces
not containing $c_0$ \cite{castmorestud} and, of course, all their
complemented subspaces. The construction of the space $\Omega$ above has been modified in \cite{castmoresua} to show that for every subspace $H\subset c_0$ there is an exact sequence
$$
\begin{CD}
 0 @>>> C(\omega^\omega)  @>>> \Omega_H @>>> c_0 @>>> 0
\end{CD}
$$
in which the space $\Omega_H$ is not a Lindenstrauss-Pe\l czy\'nski space \cite{lindpelc}; more precisely, there is an operator $H\to \Omega_H$ that cannot be extended to the whole $c_0$.
\begin{prop} There is an isomorphically polyhedral $\mathcal L_\infty$-space that is not an $\mathscr{LP}$-space.
\end{prop}
\begin{proof}
Consider the exact sequence $0 \to C(\omega^\omega) \to \Omega \to c_0 \to
0$ with strictly singular quotient constructed above. Since every quotient of $c_0$ is isomorphic to a
subspace of $c_0$, we can consider that there is an embedding
$u_H: c_0/H  \to c_0$. The pull-back sequence $0 \to
C(\omega^\omega) \to P_H \stackrel{p}\to c_0/H \to 0 $
also has strictly singular quotient map. We form the commutative
diagram

$$
\begin{CD}
&&&& 0 &=& 0\\
 &&&& @AAA @AAA\\
0@>>> C(\omega^\omega)@>>> P_H@>p>>c_0/H@>>>0\\
 &&\Vert&& @AAA @AAtA\\
 0@>>> C(\omega^\omega) @>>> \Omega_H @>>Q> c_0@>>> 0\\
 &&&& @AjAA @AAiA\\
&&&& H &=& H\\
 &&&& @AAA @AAA\\
&&&& 0 && 0 \end{CD}$$

to show, exactly as in \cite{lindpelc} that $\Omega_H$ is not an $\mathscr{LP}$-space since $j$ cannot be extended to $c_0$ through $i$. The space $\Omega_H$ has been obtained from a pull-back diagram
$$
\begin{CD}
 0 @>>> c_0(C(\omega^N))  @>>> \Omega @>>> c_0 @>>> 0\\
&&@|@AAA @AA{u_H}A\\
0@>>> C(\omega^\omega)@>>> P_H@>p>>c_0/H@>>>0\\
 &&\Vert&& @AAA @AAtA\\
 0@>>> C(\omega^\omega) @>>> \Omega_H @>>Q> c_0@>>> 0,
 \end{CD}
$$and thus it is a subspace of $\Omega\oplus c_0$, hence isomorphically polyhedral.\end{proof}

\end{document}